\documentclass{amsart}

\usepackage{
 amsmath, 
 amsxtra, 
 amsthm, 
 amssymb, 
 booktabs,
 comment,
 longtable,
 mathrsfs, 
 mathtools, 
 multirow,
 stmaryrd,
 tikz-cd, 
 bbm,
 xr,
 color,
 xcolor}
 \usepackage[bbgreekl]{mathbbol}
\usepackage[all]{xy}
\usepackage[nobiblatex]{xurl}
\usepackage{hyperref}

\newtheorem{theorem}{Theorem}[section]
\newtheorem{lemma}[theorem]{Lemma}

\newtheorem{proposition}[theorem]{Proposition}
\newtheorem{corollary}[theorem]{Corollary}
\newtheorem{defn}[theorem]{Definition}

\newtheorem{lthm}{Theorem} 

\theoremstyle{remark}
\newtheorem{remark}[theorem]{Remark}

\newcommand{\fp}{\mathfrak{p}}
\newcommand{\fq}{\mathfrak{q}}

\newcommand{\ord}{\mathrm{ord}}
\newcommand{\ZZ}{\mathbb{Z}}

\newcommand{\QQ}{\mathbb{Q}}

\newcommand{\Zp}{\ZZ_p}

\newcommand{\cH}{\mathcal{H}}
\newcommand{\cL}{\mathcal{L}}
\newcommand{\Mlog}{M_{\log}}
\newcommand{\Mlogp}{M_{\log,\fp}}
\newcommand{\Mlogq}{M_{\log,\fq}}
\definecolor{Green}{rgb}{0.0, 0.5, 0.0}

\newcommand{\fr}{\mathfrak{r}}

\renewcommand{\aa}{\mathfrak{a}}
\newcommand{\cc}{\mathfrak{c}}
\newcommand{\bb}{\mathfrak{b}}
\newcommand{\dd}{\mathfrak{d}}
\newcommand{\Qp}{\QQ_p}
\newcommand{\Laa}{\cL_{\alpha\alpha}}
\newcommand{\Lab}{\cL_{\alpha\beta}}
\newcommand{\Lba}{\cL_{\beta\alpha}}
\newcommand{\Lbb}{\cL_{\beta\beta}}

\newcommand{\Lss}{\cL_{\sharp\sharp}}
\newcommand{\Lsf}{\cL_{\sharp\flat}}
\newcommand{\Lfs}{\cL_{\flat\sharp}}
\newcommand{\Lff}{\cL_{\flat\flat}}

\newcommand{\cyc}{{\mathrm{cyc}}}

\newcommand{\sL}{\mathscr{L}}
\newcommand{\cO}{\mathcal{O}}

\DeclareMathSymbol\dDelta  \mathord{bbold}{"01}
\usepackage[utf8]{inputenc}
\numberwithin{equation}{section}
\title[$p$-adic Artin formalism at supersingular primes]{Artin formalism for $p$-adic $L$-functions of modular forms at non-ordinary primes}

\author[A.~Lei]{Antonio Lei}
\address{Antonio Lei\newline Department of Mathematics and Statistics\\University of Ottawa\\
150 Louis-Pasteur Pvt\\
Ottawa, ON\\
Canada K1N 6N5}
\email{antonio.lei@uottawa.ca}



\begin{document}
\begin{abstract} Let $p$ be an odd prime number. Let $f$ be a normalized Hecke eigen-cuspform that is non-ordinary at $p$. Let $K$ be an imaginary quadratic field in which $p$ splits. We study the Artin formalism for the two-variable signed $p$-adic $L$-functions attached to $f$ over $K$. In particular, we give evidence of a prediction made by Castella--Ciperiani--Skinner--Sprung.
\end{abstract}

\maketitle

\section{Introduction}\label{sec:intro}
Throughout, we fix an odd prime number $p$. Let $f=\sum_{n\ge1} a_nq^n$ be a normalized Hecke eigen-cuspform of weight $2$, level $N$ and nebentype $\epsilon$. We assume that $p$ does not divide $N$. We let $F$ be the completion of a number field that contains $a_n$ for all $n\ge1$ at a fixed prime above $p$. Its ring of integers is denoted by $\cO$.

We assume that $v_p(a_p)\ne 0$ (that is, $f$ is non-ordinary at $p$). Here, $v_p$ is the $p$-adic valuation on $\overline{F}$ normalized by $v_p(p)=1$. Let $\alpha$ and $\beta$ be the two roots of the Hecke polynomial of $f$ at $p$
\[
X^2-a_pX+\epsilon(p)p.
\]
By enlarging $F$ if necessary, we assume that $\alpha,\beta\in F$.

Let $\Gamma$ be the Galois group of the cyclotomic $\Zp$-extension of $\QQ$. For each $\lambda\in\{\alpha,\beta\}$, it follows from the works of Amice-V\'elu \cite{amicevelu75} and Vi\v{s}ik \cite{visik76} that there exists an $F$-valued distribution on $\Gamma$, which we denote by $\cL_\lambda$, that interpolates the complex $L$-values of $f$ twisted by Dirichlet characters factoring through $\Gamma$. We shall refer to $\cL_\lambda$ the (one-variable) $p$-adic $L$-function of $f$ associated with $\lambda$.

When $a_p=0$, Pollack \cite{pollack03} showed that the distributions $\cL_\lambda$ can be decomposed into linear combinations of bounded measures:
\begin{equation}
    \cL_\lambda=\log^+\cL_++\lambda\log^-\cL_-,\label{eq:pollack}
\end{equation}
where $\cL_\pm\in\cO\llbracket \Gamma\rrbracket$ and $\log^\pm$ are certain explicit logarithmic distributions. The elements $\cL_\pm$ are commonly refereed as the plus and minus $p$-adic $L$-functions of $f$. They are utilized in the formulation of Iwasawa main conjectures, as studied in \cite{kobayashi03,pollackrubin04}.

For general $a_p$ (that is, $v_p(a_p)>0$, but not necessarily $a_p=0$), the decomposition \eqref{eq:pollack} has been generalized by Sprung \cite{sprung09} (for elliptic curves) and by the present author together with Loeffler and Zerbes in \cite{LLZ0,LLZ3} (for general modular forms). More specifically, it has been shown that there exist $\cL_\sharp,\cL_\flat\in\cO\llbracket\Gamma\rrbracket$ such that
\begin{equation}
\begin{bmatrix}
    \cL_\alpha&\cL_\beta
\end{bmatrix}=
\begin{bmatrix}
    \cL_\sharp&\cL_\flat
\end{bmatrix}\Mlog,
\label{eq:sprung}    
\end{equation}
where $\Mlog$ is an explicit $2\times 2$ matrix. We review the construction of $\Mlog$ in \S\ref{S:log} below.

In \cite{skinnerurbanmainconj}, Skinner--Urban proved the Iwasawa main conjecture for $p$-ordinary modular forms by passing to Iwasawa main conjectures of Hida families of modular forms over the $\Zp^2$-extension of an imaginary quadratic field $K$ where $p$ splits. One of the crucial ingredients in their proof is the Artin formalism of two-variable $p$-adic $L$-functions, which allows them to pass from the arithmetic objects over the $\Zp^2$-extension of $K$ attached to $f$, back to the counterparts over the cyclotomic $\Zp$-extension of $\QQ$. In recent years, tremendous progress towards the Iwasawa main conjecture in the non-ordinary case has been made by several authors; see, in particular, \cite{wan14, sprung16, CCSS}. As in the ordinary case, the Artin formalism plays a crucial role. We describe the relevant two-variable objects below.

Let $K$ be an imaginary quadratic field where $p$ splits into two primes $\fp$ and $\fq$. Let $G$ be the Galois group of the $\Zp^2$-extension of $K$. Loeffler \cite{loeffler13} constructed four two-variable $p$-adic $L$-functions $\cL_{\lambda\mu}$, where $\lambda,\mu\in\{\alpha,\beta\}$. These $p$-adic $L$-functions are $F$-valued distributions on $G$, which interpolate the complex $L$-values of $f/K$ (the base-change of $f$ to $K$) twisted by finite Hecke characters factoring through $G$.

For $\fr\in\{\fp,\fq\}$, let $\overline\fr$ denote the complex conjugation of $\fr$. Let $\Gamma_\fr$ denote the Galois group of the $\Zp$-extension of $K$ where $\overline\fr$ is unramified. When $a_p=0$, Loeffler showed that the two-variable $p$-adic $L$-functions $\cL_{\lambda\mu}$ satisfy a decomposition similar to\eqref{eq:pollack}:
\[
\cL_{\lambda\mu}=\log_\fp^+\log_\fq^+\cL_{++}+\mu\log_\fp^+\log_\fq^-\cL_{+-}+\lambda\log_\fp^-\log_\fq^+\cL_{-+}+\lambda\mu\log_\fp^-\log_\fq^-\cL_{--},
\]
where $\cL_{\pm\pm}\in\cO\llbracket G\rrbracket$ and $\log^\pm_\fr$ are distributions on $\Gamma_\fr$, generalizing Pollack's $\log^\pm$. (In \cite{loeffler13}, the normalization is slightly different from ours, and it was only proved that $
\cL_{\pm\pm}\in \cO\llbracket G\rrbracket\otimes_\cO F$; the integrality statement follows from an appropriate choice of period.)

The author of the present article generalized Loeffler's decomposition to the general $a_p$ case in \cite{lei14}, proving a two-variable version of \eqref{eq:sprung}:
   \begin{equation}\label{eq:lei}
    \begin{bmatrix}
        \Laa & \Lab\\
        \Lba & \Lbb
    \end{bmatrix}=\Mlogp^T
    \begin{bmatrix}
        \Lss & \Lsf\\
        \Lfs & \Lff
    \end{bmatrix}\Mlogq   
   \end{equation}
    where $\cL_{\bullet\star}\in\cO\llbracket G\rrbracket$ for $\bullet,\star\in\{\sharp,\flat\}$, $\Mlogp$ and $\Mlogq$ are defined as in Definition~\ref{def:log-twist}.

We now describe results on the Artin formalism.
Identifying the Galois group of the cyclotomic $\Zp$-extension of $K$ with $\Gamma$, each $\cL_{\lambda,\mu}$ gives rise to an $F$-valued distribution on $\Gamma$ under the natural projection $G\rightarrow \Gamma$ (often referred to as the cyclotomic specialization). We denote this distribution by $\cL_{\lambda\mu}^\cyc$.
In \cite{BSW}, Barrera Salazar and Williams showed that under an appropriate choice of periods, when $\lambda=\mu$, the  Artin formalism holds for $\cL_{\lambda\lambda}$, that is, for $\lambda\in\{\alpha,\beta\}$,
\begin{equation}\label{eq:BSW}
\cL_{\lambda\lambda}^\cyc=\cL_\lambda\sL_\lambda,    
\end{equation}
where $\sL_\lambda$ is the one-variable $p$-adic $L$-function of the twist of $f$ by the quadratic character attached to $K$. More precisely, it is the modular form $f^{(K)}=\sum_{n\ge1} a_n\epsilon_K(n)q^n$, where $\epsilon_K$ is the quadratic character attached to $K$. Note that since $p$ is assumed to be split completely in $K$, the Hecke polynomial of $f^{(K)}$ at $p$ coincides with that of $f$. We note that the proof of \eqref{eq:BSW} relies on passing to $p$-adic $L$-functions attached to Coleman families for which we have interpolation formulae at a larger family of characters of $\Gamma$.

In \cite[Proposition~8.4]{leipalvannan}, Loeffler's signed $p$-adic $L$-functions $\cL_{++}$ and $\cL_{--}$ also satisfy the Artin formalism, that is, under an appropriate choice of periods,
\begin{equation}
\label{eq:LP}\cL_{\bullet\bullet}^\cyc=\cL_\bullet\sL_\bullet    
\end{equation}
for $\bullet\in\{+,-\}$. The proof of \eqref{eq:LP} relies on evaluating the elements on the two sides of the equations at Dirichlet characters whose conductors are even or odd powers of $p$, depending on the choice of $\bullet$.

The first main result of the present article is the following generalization of \eqref{eq:LP}, which removes the hypothesis that $a_p=0$.

\begin{lthm}\label{thmA}
Let $f=\sum_{n\ge1}a_nq^n$ be a normalized Hecke eigen-cuspform of weight $2$ and level $N$. Assume that $p$ is an odd prime such that $p\nmid N$ and $v_p(a_p)>\frac1p$. For $\bullet\in\{\sharp,\flat\}$, the following equation holds
    \[
\cL_{\bullet\bullet}^\cyc=\cL_\bullet\sL_\bullet.
\]
\end{lthm}
Unlike the $a_p=0$ case, where $\cL_{++}$ and $
\cL_{--}$ satisfy interpolation formulae at infinitely many finite characters of $\Gamma$, the only character where such formulae are available is the trivial character when $a_p\ne0$. In particular, the proof of \eqref{eq:LP} cannot be applied to the general case. Instead, our proof of Theorem~\ref{thmA} relies on a careful analysis of the $p$-adic valuations of the entries of $\Mlog$ evaluated at finite characters of $\Gamma$, combined with elementary manipulations on the equations \eqref{eq:sprung}, \eqref{eq:lei} \eqref{eq:BSW}. In \S\ref{S:ap0}, we specialize to the case where $a_p=0$. We show how the proof of Theorem~\ref{thmA} is simplified in this special case and how this provides an alternative proof of \eqref{eq:LP} without directly utilizing interpolation formulae of the plus and minus $p$-adic $L$-functions.

We now turn our attention to the $p$-adic $L$-functions of ``mixed type". In \cite[Proposition~3.5 and (3.6)]{CCSS} it is predicted that
\[
\cL_{\alpha\beta}^\cyc\stackrel?=\cL_{\alpha\beta}^\cyc\stackrel?=\frac{1}{2}\left(\cL_\alpha\sL_\beta+\cL_\alpha\sL_\beta\right),
\]
and that
\[
\Lsf^\cyc\stackrel?=\Lfs^\cyc\stackrel?=\frac{1}{2}\left(\cL_\sharp\sL_\flat+\cL_\flat\sL_\sharp\right).
\]
The authors of \cite{CCSS} assert that these predictions should follow from the existence of two-variable Beilinson--Flach elements (similar to those considered in \cite{castellaJLMS,BL-IMRN,BLLV}, which, in turn, are based on the works \cite{LLZ1,LLZ2,KLZ1,KLZ2}), together with a hypothetical comparison with the Beilinson--Kato elements constructed in \cite{kato04}.
We prove that as a consequence of Theorem~\ref{thmA}, a weaker form of these predictions holds.

\begin{lthm}\label{thmB}
Let $f$ be as in Theorem~\ref{thmA}. The following equations hold
    \begin{align*}
        \cL_{\alpha\beta}^\cyc+\cL_{\alpha\beta}^\cyc&=\cL_{\alpha}\sL_\beta+\cL_{\beta}\sL_{\alpha},\\
        \cL_{\sharp\flat}^\cyc+\cL_{\flat\sharp}^\cyc&=\cL_{\sharp}\sL_\flat+\cL_{\flat}\sL_{\sharp}.
    \end{align*}
\end{lthm}

We note that the condition of $v_p(a_p)>\frac{1}{p}$ is required in our calculations of $p$-adic valuations of the entries of $\Mlog$ evaluated at finite characters of $\Gamma$. This condition was also featured in our previous work \cite{LLZ3}. It should be possible to relax this condition by a more arduous analysis of these valuations.

\subsection*{Acknowledgement}
We thank Florian Sprung for the interesting discussions during the preparation of the article. 
The author's research is supported by the NSERC Discovery Grant Program RGPIN-2020-04259.

\section{The logarithmic matrix}\label{S:log}
We fix the notation that will be used in the remainder of the article.
\begin{defn}
    Let $\cH$ be the ring of power series in $F\llbracket X\rrbracket$ that converge on the open unit disk.
\end{defn}
The ring $\cH$ is equipped with the standard sup norm. We identify $\cH$ with the set of $F$-valued distributions on $\Gamma$ and $1+X$ is regarded as a topological generator of $\Gamma$. The subset $\cO\llbracket\Gamma\rrbracket$ is identified with the $\cO$-valued measures on $\Gamma$.

\begin{defn}\label{def:Cs}
We define the following $2\times 2$ matrices 
$$
C=\begin{bmatrix}
    a_p&1\\-\epsilon(p)p&0
\end{bmatrix},\quad A=\begin{bmatrix}
    -1&-1\\ \beta&\alpha
\end{bmatrix}.
$$
For an integer $n\ge1$, we define 
$$C_n=\begin{bmatrix}
    a_p&1\\-\epsilon(p)\Phi_n&0
\end{bmatrix},\quad M_n=C_1\cdots C_n C^{-n-2}A,$$
where $\Phi_n$ is the $p^n$-th cyclotomic polynomial in $1+X$.
\end{defn}

As shown in \cite[Theorem~1.5]{lei14}, the entries of the matrices $M_n$ converge under the sup norm of $\cH$ as $n\rightarrow\infty$.
\begin{defn}\label{def:log-twist}
    We define the $2\times2$ matrix $\Mlog$ as the limit $\displaystyle\lim_{n\rightarrow\infty} M_n$. 

For $\fr\in\{\fp,\fq\}$, recall that $\Gamma_\fr$ is the Galois group of the $\Zp$-extension of $K$ where $\overline\fr$ is unramified. We choose a topological generator $\gamma_\fr$ of $\Gamma_\fr$ so that $\gamma_\fr$ is sent to $1+X$ under the natural projection map $G\rightarrow\Gamma$. We define $M_{\log,\fr}$ as the $2\times 2$ matrix with entries in $F\llbracket \gamma_\fr-1\rrbracket$ obtained from replacing $X$ in $\Mlog$ by $\gamma_\fr-1$.
\end{defn}

\begin{remark}\label{rk:det}
    In the proof of \cite[Theorem~1.5]{lei14}, it has been shown that $\det(\Mlog)$ is equal to, up to a non-zero scalar in $F$, $\log_p(1+X)/X$, where $\log_p$ denotes the $p$-adic logarithm.
\end{remark}

\begin{defn}
    Let $M=\begin{bmatrix}
        a&b\\c&d
    \end{bmatrix}$ be a $2\times 2$ matrix defined over $\overline F$. We define the matrix of $p$-adic valuation 
    \[
    \ord_p(M)=\begin{bmatrix}
        \ord_p(a)&\ord_p(b)\\
        \ord_p(c)&\ord_p(d)
    \end{bmatrix}.
    \]
    Note that we write $\ord_p(0)=\infty$.
\end{defn}
Let $M=\begin{bmatrix}
        a&b\\c&d
    \end{bmatrix}$ and $M'=\begin{bmatrix}
        a'&b'\\c'&d'
    \end{bmatrix}$ be two matrices defined over $\overline F$. As explained in \cite[\S4.1]{sprung13}, we may compute $\ord_p(MM')$ via the formula
\begin{equation}\label{eq:formula}
\ord_p(M)\ord_p(M')=\begin{bmatrix}
        \min(\ord_p(aa'),\ord_p(bc'))&\min(\ord_p(ab'),\ord_p(bd'))\\
        \min(\ord_p(a'c),\ord_p(c'd))&\min(\ord_p(b'c),\ord_p(dd'))
    \end{bmatrix}
\end{equation}
provided that $\ord_p(aa')\ne\ord_p(bc')$, $\ord_p(ab')\ne\ord_p(bd')$, $\ord_p(a'c)\ne\ord_p(c'd)$ and $\ord_p(b'c)\ne\ord_p(dd')$.
\begin{defn}
    For each integer $n\ge1$, we define
$\varpi_n=\zeta_{p^n}-1$, where $\zeta_{p^n}$ is a primitive $p^n$-th root of unity. Furthermore, we define the rational numbers
\begin{align*}
t_n^+&=\sum_{i=1}^{\lfloor n/2\rfloor}\frac{1}{p^{2i}},\\
t_n^-&=\sum_{i=1}^{\lfloor n/2\rfloor}\frac{1}{p^{2i-1}}.
\end{align*}
\end{defn}

\begin{remark}\label{rk:compare}
    Direct calculations show that $t_n^+<t_n^-$, and that the sequences $\{t_n^+\}_{n\ge1}$ and $\{t_n^-\}_{n\ge1}$ converge to $\frac{1}{p^2-1}$ and $\frac{p}{p^2-1}$,  respectively, as $n\rightarrow\infty$. 
\end{remark}

\begin{proposition} 
\label{prop:eval}Let $r=v_p(a_p)$ and assume that $2r>\frac1{p}$. For all integers $n\ge2$, we have:
    \[
    \ord_p\left(\Mlog(\pi_n)\right)= \begin{cases}
        \begin{bmatrix}
            r+t_n^+-(n+1)v_p(\alpha) &r+t_n^+-(n+1)v_p(\beta)\\
            t_n^- -(n+1)v_p(\alpha)&t_n^--(n+1)v_p(\beta)
        \end{bmatrix}&\text{if $n$ is even,}\\ \\
        \begin{bmatrix}
    t_n^--(n+1)v_p(\alpha)&t_n^--(n+1)v_p(\beta)\\        
            r+t_n^+-(n+1)v_p(\alpha)&r+t_n^+-(n+1)v_p(\beta)
        \end{bmatrix}&\text{otherwise.}
    \end{cases}
    \]
\end{proposition}
\begin{proof}
As in \cite[proof of Theorem~1.5]{lei14}, we have
    \begin{align*}
    \Mlog(\pi_n)&= C_1\cdots C_n (\varpi_n)C^{-n-2}A\\
    &= C_1\cdots C_{n-1}(\varpi_n)\begin{bmatrix}
        a_p&1\\0&0
    \end{bmatrix}\begin{bmatrix}
        -\alpha^{-n-2}&-\beta^{-n-2}\\\beta\alpha^{-n-2}&\alpha\beta^{-n-2}
    \end{bmatrix}\\
    &=C_1\cdots C_{n-1}(\varpi_n)\begin{bmatrix}
        -\alpha^{-n-1}&-\beta^{-n-1}\\0&0
    \end{bmatrix}.
    \end{align*}    
    Furthermore, as in \cite[proof of Lemma~4.5]{LLZ3},  for all $1\le m<n$, the following equation holds
    \[
    \ord_p(C_m(\pi_n))=\begin{bmatrix}
        r&0\\\frac{1}{p^{n-m}}&\infty
    \end{bmatrix}.
    \]
  Therefore, as matrices of $p$-adic valuations, we have
    \[
    \ord_p(\Mlog(\varpi_n))=\begin{bmatrix}
        r&0\\\frac{1}{p^{n-1}}&\infty
    \end{bmatrix}\cdots \begin{bmatrix}
        r&0\\\frac{1}{p}&\infty
    \end{bmatrix}\begin{bmatrix}
        -(n+1)v_p(\alpha)&-(n+1)v_p(\beta)\\\infty&\infty
    \end{bmatrix}.
    \]
Hence, it suffices to prove that the following equation of matrices of $p$-adic valuation holds:
\begin{equation}
\label{eq:claim}
\begin{bmatrix}
        r&0\\\frac{1}{p^{n-1}}&\infty
    \end{bmatrix}\cdots \begin{bmatrix}
        r&0\\\frac{1}{p}&\infty
    \end{bmatrix}\begin{bmatrix}
        a&b\\\infty&\infty
    \end{bmatrix}= 
    \begin{cases}
        \begin{bmatrix}
            r+t_n^++a &r+t_n^++b\\
            t_n^- +a&t_n^-+b
        \end{bmatrix}&\text{if $n$ is even,}\\ \\
        \begin{bmatrix}
    t_n^-+a&t_n^-+b\\        
            r+t_n^++a&r+t_n^++b
        \end{bmatrix}&\text{otherwise.}\end{cases}
\end{equation}
for any $a,b\in\QQ$.

    A direct calculation shows that \eqref{eq:claim} holds for $n\in\{2,3\}$, thanks to \eqref{eq:formula} and our assumption that $2r>\frac1{p}$. Suppose that \eqref{eq:claim} holds for $n=2k-1$. Then,
    \begin{align*}
    \begin{bmatrix}
        r&0\\\frac{1}{p^{2k-1}}&\infty
    \end{bmatrix}\cdots \begin{bmatrix}
        r&0\\\frac{1}{p}&\infty
    \end{bmatrix}\begin{bmatrix}
        a&b\\\infty&\infty
    \end{bmatrix}&=\begin{bmatrix}
        r&0\\\frac{1}{p^{2k-1}}&\infty
    \end{bmatrix} \begin{bmatrix}t_{2k-1}^-+a&t_{2k-1}^-+b\\        
            r+t_{2k-1}^++a&r+t_{2k-1}^++b
        \end{bmatrix}\\\\
        &=\begin{bmatrix}
        r+t_{2k-1}^++a&r+t_{2k-1}^++b\\r+t_{2k-1}^-+\frac{1}{p^{2k-1}}+a&r+t_{2k-1}^-+\frac{1}{p^{2k-1}}+b
        \end{bmatrix}\\ \\
        &=\begin{bmatrix}
        r+t_{2k}^++a&r+t_{2k}^++b\\r+t_{2k}^-+a&r+t_{2k}^-+b
        \end{bmatrix},
    \end{align*}
    where we have utilized \eqref{eq:formula} and the fact that $t_{2k-1}^-<t_{2k-1}^+$ as discussed in Remark~\ref{rk:compare}. In particular, \eqref{eq:claim} holds for $n=2k$. We can then show that the same is true for $n=2k+1$ by a similar calculation. Thus, \eqref{eq:claim} holds for all $n\ge 2$ by induction, which concludes the proof of the proposition.    
\end{proof}

\begin{corollary}\label{cor:nonzero}
    All four entries in the matrix $\Mlog$ are non-zero elements of $\cH$.
\end{corollary}
\begin{proof}
    This follows from Proposition~\ref{prop:eval} since each entry of $\ord_p\left(\Mlog(\varpi_n)\right)$ is not identically $\infty$ as $n$ varies.
\end{proof}

\section{Proof of Theorem~\ref{thmA}}
\label{S:A}
To simplify the notation, we write
    \[
    \Mlog=\begin{bmatrix}
        \aa&\bb\\
        \cc&\dd
    \end{bmatrix}
    \]
    from now on.
The following lemma is a straightforward consequence of some of the results mentioned in the Introduction to this article.

\begin{lemma}We have
\[
    \begin{bmatrix}
        \cL_\alpha\sL_\alpha & \Lab^\cyc\\
        \Lba^\cyc & \cL_\beta\sL_\beta
    \end{bmatrix}=\Mlog^T
    \begin{bmatrix}
        \Lss^\cyc & \Lsf^\cyc\\
        \Lfs^\cyc & \Lff^\cyc
    \end{bmatrix}\Mlog.
    \]
    In other words,
    \begin{align}
        \cL_\alpha\sL_\alpha&=\aa^2\Lss^\cyc+\cc^2\Lff^\cyc+\aa\cc(\Lsf^\cyc+\Lff^\cyc),\label{eq:aa}\\
        \cL_\beta\sL_\beta&=\bb^2\Lss^\cyc+\dd^2\Lff^\cyc+\bb\dd(\Lsf^\cyc+\Lff^\cyc), \label{eq:bb}\\
        \Lab^\cyc&=\aa\bb\Lss^\cyc+\bb\cc\Lfs^\cyc+\aa\dd\Lsf^\cyc+\cc\dd\Lff^\cyc,\label{eq:ab}\\
        \Lba^\cyc &= \aa\bb\Lss^\cyc+\aa\dd\Lfs^\cyc+\bb\cc\Lsf^\cyc+\cc\dd\Lff^\cyc.\label{eq:ba}
    \end{align}
\end{lemma}
\begin{proof}
    This follows from combining \eqref{eq:BSW} and the cyclotomic specialization of \eqref{eq:lei}, which sends both $\Mlogp$ and $\Mlogq$ to $\Mlog$.
\end{proof}

Consider $\bb\dd\times$\eqref{eq:aa}$-\aa\cc\times$\eqref{eq:bb}:
\begin{equation}
\bb\dd\cL_\alpha\sL_\alpha-\aa\cc\cL_\beta\sL_\beta=(\aa^2\bb\dd-\bb^2\aa\cc)\Lss^\cyc+(\cc^2\bb\dd-\dd^2\aa\cc)\Lff^\cyc.
\label{eq:key}  
\end{equation}

The decomposition \eqref{eq:sprung} and its counterpart for the quadratic twist $f^{(K)}$ say that
    \begin{equation}
    \begin{bmatrix}
        \cL_\alpha&\cL_\beta
    \end{bmatrix}=\begin{bmatrix}
        \cL_\sharp&\cL_\flat
    \end{bmatrix}\begin{bmatrix}
        \aa&\bb\\
        \cc&\dd
    \end{bmatrix},\quad
    \begin{bmatrix}
        \sL_\alpha&\sL_\beta
    \end{bmatrix}=\begin{bmatrix}
        \sL_\sharp&\sL_\flat
    \end{bmatrix}\begin{bmatrix}
        \aa&\bb\\
        \cc&\dd
    \end{bmatrix}.
    \label{eq:sprungsprung}    
    \end{equation}
    Therefore, the left-hand side of \eqref{eq:key} can be rewritten as:
\[
(\aa^2\bb\dd-\bb^2\aa\cc)\cL_\sharp\sL_\sharp+(\cc^2\bb\dd-\dd^2\aa\cc)\cL_\flat\sL_\flat.
\]
In particular, we see that both sides of \eqref{eq:key} are divisible by $\aa\dd-\bb\cc$ (which is a non-zero element of $\cH$ by Remark~\ref{rk:det}). Therefore, as $\cH$ is an integral domain, dividing by $\aa\dd-\bb\cc$ on both sides of \eqref{eq:key} gives
\begin{equation}
\aa\bb(\Lss^\cyc-\cL_\sharp\sL_\sharp)=\cc\dd(\Lff^\cyc-\cL_\flat\sL_\flat).
\label{eq:key2}    
\end{equation}

Suppose that the quantities on the two sides of \eqref{eq:key2} are non-zero. Consequently, for $\bullet\in\{\sharp,\flat\}$, the difference $\cL_{\bullet\bullet}^\cyc-\cL_\bullet\sL_\bullet$ is a non-zero element of $\cO\llbracket\Gamma\rrbracket$. Let $\pi$ be a uniformizer of $\cO$. By Weierstrass' preparation theorem,
\[
\cL_{\bullet\bullet}^\cyc-\cL_\bullet\sL_\bullet=u_\bullet \pi^{\mu_\bullet} P_\bullet,
\]
for some $u_\bullet\in\cO\llbracket \Gamma\rrbracket^\times$,  $\mu_\bullet\in\ZZ_{\ge 0}$, and a distinguished polynomial $P_\bullet\in\cO[X]$.

Let $n\ge2$ be an even integer. Evaluating both sides of \eqref{eq:key2} at $X=\varpi_n$, we deduce that
\[
(\cL_{\bullet\bullet}^\cyc-\cL_\bullet\sL_\bullet)(\varpi_n)=\frac{\mu_\bullet} e+\frac{\deg(P_\bullet)}{p^{n-1}(p-1)}
\]
for $n\gg0$, where $e$ is the ramification index of $F/\Qp$. Combining this equation with Proposition~\ref{prop:eval}, gives
\begin{equation}\label{eq:even-big}
2(r+t_n^+)+\frac{\mu_\sharp} e+\frac{\deg(P_\sharp)}{p^{n-1}(p-1)}=
2t_n^-+\frac{\mu_\flat} e+\frac{\deg(P_\flat)}{p^{n-1}(p-1)}.    
\end{equation}
Let $n\rightarrow\infty$, we deduce from Remark~\ref{rk:compare} that
\begin{equation}
2\left(r-\frac{1}{p+1}\right)=\frac{\mu_\flat-\mu_\sharp} e.
\label{eq:even-limit}    
\end{equation}

Taking $n\ge3$ to be odd integers instead, we deduce similarly:
\begin{equation}
2\left(r-\frac{1}{p+1}\right)=\frac{\mu_\sharp-\mu_\flat} e.
\label{eq:odd-limit}    
\end{equation}
Equations \eqref{eq:even-limit} and \eqref{eq:odd-limit} then imply that
\begin{equation}
r=\frac{1}{p+1},\quad \mu_\sharp=\mu_\flat.    
\end{equation}
This contradicts our hypothesis that $r>\frac{1}{p}$, which concludes the proof of Theorem~\ref{thmA}.

\section{Proof of Theorem~\ref{thmB}}
We now turn our attention to the proof of Theorem~\ref{thmB}, which turns out to be a straightforward consequence of Theorem~\ref{thmA}.

Thanks to \eqref{eq:sprungsprung}, we deduce from \eqref{eq:aa}:
\[
\aa^2\cL_\sharp\sL_\sharp+\cc^2\cL_\flat\sL_\flat+\aa\cc(\cL_\sharp\sL_\flat+\cL_\flat\sL_\sharp)=\aa^2\Lss^\cyc+\cc^2\Lff^\cyc+\aa\cc(\Lsf^\cyc+\Lff^\cyc).
\]
Therefore, it follows from Theorem~\ref{thmA} that
\[
\aa\cc(\cL_\sharp\sL_\flat+\cL_\flat\sL_\sharp)=\aa\cc(\Lsf^\cyc+\Lff^\cyc).
\]
Corollary~\ref{cor:nonzero} tells us that $\aa\cc\ne0$. Thus,
\begin{equation}\label{eq:sharpflat}
\cL_\sharp\sL_\flat+\cL_\flat\sL_\sharp=\Lsf^\cyc+\Lff^\cyc,
\end{equation}
proving the second assertion of Theorem~\ref{thmB}.

If we take the sum of the equations \eqref{eq:ba} and \eqref{eq:ab}, we obtain the equality
\begin{equation}
\Lab^\cyc+\Lba^\cyc=2\aa\bb\Lss^\cyc+(\aa\dd+\bb\cc)(\Lsf^\cyc+\Lfs^\cyc)+2\cc\dd\Lff^\cyc.
\label{eq:mixed}    
\end{equation}
The right-hand side of \eqref{eq:mixed} is equal to
\[
2\aa\bb\cL_\sharp\sL_\sharp+(\aa\bb+\bb\cc)(\cL_\sharp\sL_\flat+\cL_\flat\sL_\sharp)+2\cc\dd\cL_\flat\sL_\flat,
\]
thanks to Theorem~\ref{thmA} and \eqref{eq:sharpflat}. This can be rewritten as
\[
\cL_\alpha\sL_\beta+\cL_\beta\sL_\alpha
\]
after utilizing once again \eqref{eq:sprungsprung}. Hence, the first assertion of Theorem~\ref{thmB} follows.

\section{The $a_p=0$ case}\label{S:ap0}
Assuming $a_p=0$, we discuss how our proof of Theorem~\ref{thmA} gives an alternative proof of \eqref{eq:LP}, without directly utilizing interpolation formulae of the plus and minus $p$-adic $L$-functions. 

We recall that Pollack's plus and minus logarithms are defined as
\[
\log^+=\frac{1}{p}\prod_{m\ge 1}\frac{\Phi_{2m}}{p},\quad
\log^-=\frac{1}{p}\prod_{m\ge 1}\frac{\Phi_{2m-1}}{p},
\]
where $\Phi_i$ denotes the $p^i$-th cyclotomic polynomial in $1+X$, as in Definition~\ref{def:Cs}.

One can check that \eqref{eq:key2} takes the form
\begin{equation}
(\log^+)^2(\cL_{++}^\cyc-\cL_+\sL_+)=\eta\cdot(\log^-)^2(\cL_{--}^\cyc-\cL_-\sL_-).
\label{eq:key3}    
\end{equation}
for some scalar $\eta\in F^\times$. If $\theta$ is a character of $\Gamma$ that sends $X$ to $\varpi_{2m}$, $m\ge1$, then $\log^+(\theta)=0$, but $\log^-(\theta)\ne0$. Thus, \eqref{eq:key3} implies that
\[
(\cL_{--}^\cyc-\cL_-\sL_-)(\theta)=0.
\]
Since this holds for infinitely many $\theta$ and $\cL_{--}^\cyc-\cL_-\sL_-\in\cO\llbracket\Gamma\rrbracket$, this implies that $\cL_{--}^\cyc=\cL_-\sL_-$. Similarly, we can show that $\cL_{++}^\cyc=\cL_+\sL_+$ by considering the characters that send $X$ to $\varpi_{2m-1}$, $m\ge1$. Thus, \eqref{eq:LP} follows.

When $a_p\ne0$, the vanishing of $(\cL_{\bullet\bullet}^\cyc-\cL_\bullet\sL_\bullet)(\theta)$ does not hold. Nonetheless, our explicit construction of $\Mlog$ allows us to extract information on its $p$-adic valuation via Proposition~\ref{prop:eval}, which is the key input of our argument presented in \S\ref{S:A}.

\bibliographystyle{amsalpha}
\bibliography{references}

\end{document}